\documentclass[12pt]{amsart}
\usepackage{mathrsfs,eucal,amsmath,amsthm,amsfonts,amssymb,amscd,amsbsy,late
xsym,dsfont}
\usepackage[all,2cell,dvips]{xy} \UseAllTwocells \SilentMatrices

\usepackage{verbatim}

\newtheorem{thm}{Theorem}
\newtheorem{lem}[thm]{Lemma}

\newtheorem{prp}[thm]{Proposition}

\theoremstyle{remark}
\newtheorem{rmk}[thm]{Remark}
\newtheorem{ex}[thm]{Example}

\theoremstyle{definition}
\newtheorem{dfn}[thm]{Definition}

\newcommand{\INVISIBLE}[1]{}

\begin{document}
\title[Bundles trivialized by proper morphisms, II]{Vector
bundles trivialized by proper morphisms and the fundamental
group scheme, II}

\author[I. Biswas]{Indranil Biswas}

\address{School of Mathematics, Tata Institute of Fundamental
Research, Homi Bhabha Road, Bombay 400005, India}

\email{indranil@math.tifr.res.in}

\author[J. P. dos Santos]{Jo\~ao Pedro P. dos Santos}

\address{Facult\'e de Math\'ematiques, Universit\'e de Paris VI. 
4, Place Jussieu, Paris  75005, 
France}

\email{dos-santos@math.jussieu.fr}
\urladdr{http://people.math.jussieu.fr/~dos-santos}


\date{18.05.2011}

\begin{abstract}Let $X$ be a projective and smooth variety over an algebraically closed field $k$. Let $f:Y\longrightarrow X$ be a proper and surjective morphism of $k$--varieties. Assuming that $f$ is separable, we prove that the Tannakian category associated to the vector bundles $E$ on $X$ such that $f^*E$ is trivial is equivalent to the category of representations of a finite and etale group scheme. We give a counterexample to this conclusion in the absence of separability. 
\end{abstract}

\maketitle

\section{Introduction}

The present work is a continuation of \cite{tef}, giving some
applications of the main result in \cite{tef} which throw light on
the nature of the fundamental group scheme of Nori \cite{nori} for
a smooth projective variety.

Let $X$ be a smooth projective variety over an algebraically closed
field $k$. The fundamental group scheme of $X$ is the affine group scheme obtained from the (Tannakian) category of  essentially finite vector bundles on $X$ (see
Definition \ref{dfn_essfinite}).
The main theorem of \cite{tef} says that a vector bundle
$E$ over $X$ is essentially finite
if and only if there is a proper $k$--scheme $Y$ and
a surjective morphism $f:Y\longrightarrow X$ such that
$f^*E$ is trivial. As an application
of this theorem, we prove the following :
\begin{thm}\label{thm0}
Let $X$ be a smooth and projective variety over the algebraically closed field $k$, $x_0:\mathrm{Spec}(k)\longrightarrow X$ a point, and  
$f\,:\, Y\, \longrightarrow\, X$ a proper and surjective morphism of varieties. 

\emph{(i)} The full subcategory of $\mathbf{VB}(X)$
\[
\mathcal T_Y(X)=
\left\{
\text{$V\in\mathbf{VB}(X)$\,:\, $f^*V$ is trivial}
\right\}
\]
is Tannakian. The functor $x_0^*:\mathcal T_Y(X)\longrightarrow
\text{($k$--mod)}$ is a fibre functor. 

\emph{(ii)}  Assume that $f$ is separable. Let $G(Y/X)$ denote the affine group scheme obtained from $\mathcal T_Y(X)$ and $x_0^*$.  Then $G(Y/X)$ is finite and etale.

\emph{(iii)} If the separability assumption on $f$ is removed, then there exists a counterexample to the conclusion in \emph{(ii)} in which $G(Y/X)$ is not a finite group scheme. 
\end{thm}

Part (i) of the above theorem is routine, see Lemma \ref{17.05.2011--1}. Part (ii) is the subject of Theorem \ref{main_separable1}, while the counterexample alluded to in (iii) is produced in section \ref{18.05.2011--1}.  

\section{Preliminaries}

\subsection{Notation and terminology}\label{terminology}
Throughout $k$ will stand for an algebraically closed field; by a
variety we mean an integral scheme of finite type over $k$.

Let $V$ be a normal
variety. Its field of rational functions will be denoted by $R(V)$.
We will let $\mathrm{Val}(V)$ denote the set of discrete valuations
of $R(V)$ associated to $V$: a discrete valuation 
$v:R(V)\longrightarrow \mathds Z\cup\{\infty\}$ belongs to
$\mathrm{Val}(V)$ if and only if there exists a point $\xi$ of
codimension one in $V$ such that $\mathcal O_{V,\xi}=\{\varphi\in
R(V):\,v(\varphi)\ge 0\}$.

Given a finite extension of fields $L/K$ and a set of discrete
valuations $S$ of $K$, we say that $L$ is unramified above $S$ if
for each discrete valuation $v$ of $S$ and each prolongation $w$ of
$v$ to $L$, the ramification index $e(w/v)=1$ and the extension of
residue fields is separable.

A dominant morphism $f:W\longrightarrow V$ between two varieties is
\emph{separable} if the extension of function fields $R(W)/R(V)$ is
\emph{separable}. (This differs from the homonymous notion defined in
[SGA1 X, Definition 1.1].)

A vector bundle over a scheme is a locally free coherent sheaf. The
category of all vector bundles on $X$ will be denoted by
$\mathbf{VB}(X)$.
If $E$ is a vector bundle over the $k$--scheme $X$, we will say that
$E$ comes from a representation of the \'etale fundamental group if
there exists a finite group $\Gamma$, a representation
\[
\rho\,:\,\Gamma\,\longrightarrow\, \mathrm{GL}_m(k)
\]
and an \'etale Galois covering of
group $\Gamma$, $Y\longrightarrow X$, such that
\[
E\,\cong\, Y\times^\Gamma k^{\oplus m}\, .
\]
(For the general  definition of the contracted product of a torsor and a representation, see e.g. \cite[I 5.8, 5.14]{J}.)

Given an affine group scheme $G$ over $k$, we will let 
$\mathrm{Rep}(G)$ denote the category of all finite dimensional 
representations of $G$ \cite[Ch. 3]{waterhouse}. A morphism of affine 
group schemes $f:G\longrightarrow H$ is a quotient morphism if it is 
faithfully flat, or, equivalently, if the homomorphism induced on the 
function rings is injective \cite[Ch. 14]{waterhouse}.

If $\mathcal T$ is a Tannakian category over $k$ \cite{deligne-milne} and $V\in\mathcal T$, 
we define the monodromy category of $V$ to be the smallest Tannakian 
sub--category of $\mathcal T$ containing $V$: it will be denoted by 
$\langle V ;\mathcal T\rangle_\otimes$.

Consider the neutral Tannakian category $\mathrm{Rep}(G)$
over $k$. For any $V\, \in\, \mathrm{Rep}(G)$, the 
category $\langle V; \mathrm{Rep}(G)\rangle_\otimes$ 
is equivalent to the category of representations of 
the image of the tautological homomorphism $\rho_V:G\longrightarrow 
\mathrm{GL}(V)$. This image will be called the monodromy group of $V$;
see Definition 2.5 and the remark after it in \cite{tef} 
for more information.

%%%%%%%%%%%%%%%%%%%%%%%%%%%%%%%%%%%%%%%%%%%%%%%%%%
\subsection{Vector bundles trivialized by proper and surjective
morphisms}\label{main_tef}

Let $X$ be a smooth and 
projective variety over $k$ (recall that $k$ is algebraically
closed).

\begin{dfn}[Property T]\label{propt}A vector bundle $E$ over $X$ is
said to have \emph{property (T)} if there exists a proper $k$--scheme
$Y$ together with
a surjective
(proper) morphism $f:Y\longrightarrow X$ such that the pull--back
$f^*E$
is trivial.
\end{dfn}

The main result of \cite{tef}
 relates property (T) to the more sophisticated notion of
essential finiteness.

\begin{dfn}\label{dfn_essfinite} Following Nori \cite{nori}, we say
that a vector bundle over $X$ is \emph{essentially finite} if there
exists a finite group scheme $G$, a $G$--torsor $P\longrightarrow X$
and a representation $\rho: G\longrightarrow \mathrm{GL}(V)$, such
that
\[
P\times^GV\,\cong\, E\, .
\]
The category of all essentially finite vector 
bundles over $X$ will be denoted by $\mathbf{EF}(X)$.
\end{dfn}
\begin{rmk}Every essentially finite vector bundle enjoys
property (T) as these are trivialized by a torsor under a finite
group scheme.

The category $\mathbf{EF}(X)$ is Tannakian \cite{nori}. The above definition of 
essential finiteness is not the one presented
in \cite{nori}, but a consequence of the results of that work.
\end{rmk}

\begin{thm}\cite[Theorem 1.1]{tef}\label{theorem_main}
A vector bundle $E$ over $X$ is
essentially finite if and only if it satisfies property (T).
\end{thm}

The reader is urged to read Remark \ref{parameswaran_remark} at the end of 
this text to be directed to another proof of the case where $\dim X=1$; this 
proof was suggested to us by Parameswaran and is based on \cite[\S 
6]{BP} (which contains very interesting conceptual advancements). 
Also, we indicate that recently Antei and Mehta put forward a generalisation of Theorem \ref{theorem_main} in the case where $X$ is only normal \cite{antei_mehta}.

It should be clarified that the smoothness condition on Theorem \ref{theorem_main}
cannot be dropped; this is shown by the following
example:

\begin{ex}\label{main_tef_example}
Let $X\subset\mathds P^2_k$ be the nodal cubic defined by 
$(y^2z=x^3+x^2z)$. Let 
\[
f:\mathds P^1\longrightarrow X,\quad (s:t)\mapsto (s^2t-t^3:s^3-st^2:t^3),
\]
be the birational morphism which identifies the points $(1:1)$ and 
$(-1:1)$. It is well--known that $\mathrm{Pic}^0(X)=k^*$, so that any 
line bundle $L$ of infinite order over $X$ gives a counter--example to 
the generalization of Theorem \ref{theorem_main} to the case where $X$ 
is not normal.  
\end{ex}

\subsection{The fundamental group--scheme}
Fix a $k$--rational point $x_0:\mathrm{Spec}(k)\longrightarrow X$. 
The essentially finite vector bundles with the fibre functor defined by
sending any essentially finite vector bundle $E$ to its fibre $x_0^*E$ over $x_0$
form a neutral Tannakian category \cite[Definition 2.19]{deligne-milne}. The corresponding affine 
group--scheme over $k$ \cite[Theorem 2.11]{deligne-milne} is called the \textit{fundamental group--scheme} 
\cite{nori}, \cite{nori-thesis}. This group--scheme will be denoted by
$\Pi^\mathrm{EF}(X,x_0)$.

%%%%%%%%%%%%%%%%%%%%%%%%%%%%%%%%%%%%%%%%%%%%
\section{Vector bundles trivialized by separable proper
morphisms}\label{s1}
Throughout this section, we  let $X$ stand for a projective and smooth variety and $f:Y\longrightarrow X$ for a proper surjective morphism from a (proper) variety $Y$. We  also choose a $k$--rational point $x_0:\mathrm{Spec}(k)\longrightarrow X$.

\subsection{The object of our study}\label{object_of_study} For general terminology on Tannakian categories the reader should consult \cite{deligne-milne}. 
\begin{lem}\label{17.05.2011--1}The full subcategory of $\mathbf{VB}(X)$
\[
\mathcal T_Y(X)=
\left\{
\text{$V\in\mathbf{VB}(X)$; $f^*V$ is trivial}
\right\}
\]
is Tannakian. The functor $x_0^*:\mathcal T_Y(X)\longrightarrow
\text{($k$--mod)}$ is a fibre functor. 
\end{lem}
\begin{proof}
 That $\mathcal T_Y(X)$ is stable by tensor products and
direct sums is clear. That it is an abelian category is a
consequence of the fact that all vector bundles in $\mathcal T_Y(X)$
are Nori--semistable, so that kernels and cokernels are always vector
bundles; see \cite[Corollary 2.3]{tef} and \cite[Lemma 3.6]{nori}. Using this last 
remark, it is
easy to understand why the functor $x_0^*$ is exact and faithful. As $\mathcal T_Y(X)$ has only vector bundles as objects, the rigidity axiom for a Tannakian category is promptly satisfied.
\end{proof}

The affine group scheme obtained from $\mathcal T_Y(X)$ and the fibre functor $x_0^*$
via the main Theorem of Tannakian categories \cite[Theorem 2.11]{deligne-milne} will be denoted by $G(Y/X)$ in the sequel.

\subsection{Finiteness of $G(Y/X)$  for separable morphisms}

\begin{thm}\label{main_separable1}
We assume that $f:Y\longrightarrow X$ is separable. 

(1) If the vector bundle $E$ is such that $f^*E$ is trivial, then $E$
is essentially finite and in fact comes from a representation of the \'etale
fundamental group. Moreover, the monodromy group of $E$ in the
category $\mathbf{EF}(X)$ at the point $x_0\in X(k)$ is a quotient of
a fixed finite \'etale group scheme $\Gamma^\mathrm{nr}$. (See Section
\ref{terminology} for definitions.)

(2)   The group scheme $G(Y/X)$ is finite and etale.
\end{thm}

The first step towards a proof of Theorem \ref{main_separable1} (and
also of \cite[Theorem 1.1]{tef}) is to consider the Stein
factorization of $f$:
\[
\xymatrix{Y\ar[r]^{h} \ar[dr]_f &  Y'\ar[d]^{g} \\ & X, }
\]
where $g$ is finite and $h_*(\mathcal O_Y)=\mathcal O_{Y'}$. The
latter equality implies that the morphism
$h^*:\mathbf{VB}(Y')\longrightarrow \mathbf{VB}(Y)$ is full and
faithful, so that $g^*E$ is already trivial.

\begin{dfn} Let $\varphi:V\longrightarrow X$ be a finite, surjective
and separable morphism of varieties. 
By $R(V)^{\mathrm{nr}}$ we denote the maximal unramified intermediate
extension of $R(V)/R(X)$, which is the compositum of all sub-extensions
$R$ of $R(V)/R(X)$ which are unramified over
$\mathrm{Val}(V)$. 
We let $$\varphi^\mathrm{nr}:V^\mathrm{nr}\longrightarrow X$$ denote
the normalization of $X$ in $R(V)^\mathrm{nr}$. If $R(V)/R(X)$ is
Galois of group $\Gamma$, the $\Gamma^\mathrm{nr}$ denotes the Galois
group of the  extension $R(V)^\mathrm{nr}/R(X)$.
\end{dfn}

\begin{proof}[Proof of Theorem \ref{main_separable1}] (1): That
  $E$ is
essentially finite is the content of Theorem \ref{theorem_main}. 
For the remainder, it is enough to
prove statement (1) in the theorem under the assumption that $f$ is 
finite. There is also no
loss of generality in assuming that the field extension $R(Y)/R(X)$
is Galois; let $\Gamma$ be its Galois group.

We first prove that if $\Gamma^\mathrm{nr}$ is trivial, i.e.
$f^\mathrm{nr}=\mathrm{id}_X$, then $E$ is likewise. Let $G$ be the
finite group scheme associated, by Tannakian duality, to the category
$\langle E;\mathbf{EF}(X)\rangle_\otimes$ via the point $x_0\in
X(k)$ (see Section \ref{terminology}). 
Let $P$ be the $G$--torsor associated to $E$ \cite[\S 2]{nori}; the
functor
\[
P\times^G(\bullet)\, :\, \mathrm{Rep}(G)\longrightarrow \langle
E;\mathbf{EF}(X)\rangle_\otimes
\]
induces an equivalence of monoidal categories. We denote by
$G^\mathrm{et}$ the finite \'etale group scheme of connected components
of $G$ \cite[Chapter 6]{waterhouse}. As $P$ is connected
\cite[Proposition 3, p. 87]{nori-thesis}, so is
\[
P^\mathrm{et}:=P/\ker\left( G\longrightarrow G^\mathrm{et} \right) =
P\times^GG^\mathrm{et}\, .
  \]
Since $P^\mathrm{et}\longrightarrow X$
is an \'etale morphism, it follows that $P^\mathrm{et}$
is a normal \emph{variety}.

\emph{Claim A:} The triviality of $f^*E$ implies the triviality of
the $G$--torsor \[P_Y:=P\times_XY\longrightarrow Y.\]

Let $\rho:G\longrightarrow \mathrm{GL}(V)$ be a representation of $G$
such that $E=P\times^GV$. It follows that $\rho$ is a
closed embedding and we are able to
deduce the triviality of $P_Y$ by using the triviality of
$P_Y\times^G\mathrm{GL}(V)$ together with the fact that the natural map $H^1_{\mathrm{fppf}}(Y,G)\longrightarrow H^1_\mathrm{fppf}(Y,\mathrm{GL}(V))$ is injective (as the kernel is the set of all morphisms from $Y$ to the affine scheme $\mathrm{GL}(V)/G$ \cite[p. 373, III, \S 4, 4.6]{dg}).

Let $h:Y\longrightarrow P$ be the $X$--morphism derived from an
isomorphism $P_Y\cong Y\times G$ and let $j:Y\longrightarrow
P^\mathrm{et}$ be the morphism of $X$--schemes obtained from $h$. It
is not hard to see that $j$ takes the generic point of $Y$ to the
generic point of $P^\mathrm{et}$, so $j$ gives rise to a homomorphism
of $R(X)$--fields $R(P^\mathrm{et})\longrightarrow R(Y)$. Since
$R(P^\mathrm{et})/R(X)$ is unramified above $\mathrm{Val}(X)$, we
must have $R(P^\mathrm{et})=R(X)$. As a consequence,
$P^\mathrm{et}=X$ and thus $G^\mathrm{et}$ is trivial. This means
that $G$ is a local group scheme. We will now prove the following:

\emph{Claim B:} If $G$ is local, then the existence of an
$X$--morphism $h:Y\longrightarrow P$ implies the triviality of $P$.

Let $\mathrm{Spec}(A)\subseteq X$ be an affine open and let
$\mathrm{Spec}(B)\subseteq Y$ (respectively,
$\mathrm{Spec}(S)\subseteq P$) be its pre--image in $Y$
(respectively, in $P$). We then have a homomorphism of $A$--algebras
$\eta:S\longrightarrow B$; let $S'\subseteq B$ be its image. Since
\[
\mathrm{Spec}(S)\longrightarrow \mathrm{Spec}(A)
\]
is a $G$--torsor, above any maximal ideal $\mathfrak m\subseteq A$,
there exists only one maximal ideal of $S$; the same property is
valid if we replace $S$ by $S'$. Hence, the extension of fields
defined by $S'\supseteq A$ must be \emph{purely inseparable}. Because
$R(Y)/R(X)$ is a separable extension, and $A$ is a normal ring, it
follows that $S'=A$. This allows one to construct a section $\sigma:
X\longrightarrow P$. Therefore $E$ is trivial. This proves
Claim B.

Now we treat the general case. Since
$f^\mathrm{nr}:Y^\mathrm{nr}\longrightarrow X$ is unramified above
$\mathrm{Val}(X)$, the Zariski--Nagata purity Theorem (for the
statement, see [SGA 1 X, 3.1]) permits us to conclude that
$f^\mathrm{nr}$ is \'etale (in particular $Y^\mathrm{nr}$ a smooth
projective variety over $k$). Moreover,
$f^{\mathrm{nr}}:Y^\mathrm{nr}\longrightarrow X$ is an \'etale Galois
covering of group $\Gamma^\mathrm{nr}$. Let \[g: Y\longrightarrow
Y^\mathrm{nr}\] denote the obvious morphism, we have
$g^\mathrm{nr}=\mathrm{id}_{Y^\mathrm{nr}}$. Applying what was proved
above to $Y^\mathrm{nr}$, we conclude that $f^{\mathrm{nr}*}E$ is
trivial. By \cite[Proposition 1.2]{ls}, we conclude that \[E\cong
Y^\mathrm{nr}\times^{\Gamma^\mathrm{nr}} V,\] where $V$ is a
representation of $\Gamma^\mathrm{nr}$.
This proves that the monodromy group of $E$ in $\mathbf{EF}(X)$ is a
quotient of $\Gamma^\mathrm{nr}$.

(2): The proof rests on the same sort of argument
used for the proof of (1). As in (1), we assume that $f$ is finite. 
Let \[G(Y/X):=G=\varprojlim G_i\] be the \emph{profinite} group scheme
associated to $\mathcal T_Y(X)$ via $x_0^*$; here each group $G_i$ is
finite and the transition morphisms $G_j\longrightarrow G_i$ are all
faithfully flat. (The reader unfamiliar with this sort of structure
argument will profit from \cite[3.3]{waterhouse} and
\cite[14.1]{waterhouse}.) Write
$P\longrightarrow X$ for the universal $G$--torsor \cite[\S 2]{nori}
and $P_i$ for $P\times^GG_i$. We remark that Proposition 3 on p. 87
of \cite{nori-thesis} proves that $\Gamma(P_i,\mathcal O_{P_i})=k$.
In this situation, we can find $X$--morphisms
\[
h_i:Y\longrightarrow P_i\, .
\] 
(The details of the argument are given in the proof of (1) above.)
Let $G_i^{\mathrm{et}}$ be the largest \'etale quotient of $G_i$
\cite[Ch. 6]{waterhouse}; the morphism
\[
P_i^\mathrm{et}:=P\times^GG_i^{\mathrm{et}}\longrightarrow X
\]
is finite and \'etale and the number of $k$--rational points on a fiber
equals $\mathrm{rank}\,G_i^{\mathrm{et}}$. From the surjectivity of the composition 
\[Y\longrightarrow P_i\longrightarrow P_i^\mathrm{et},\] the integers 
$\mathrm{rank}\,G_i^\mathrm{et}$ are bounded from above, so
\[G^{\mathrm{et}}:=\varprojlim
G_i^\mathrm{et}=G_{i_0}^{\mathrm{et}}\] for some $i_0$. Let
$X':=P_{i_0}^\mathrm{et}=P_i^\mathrm{et}$, it is a smooth and
projective variety and the obvious morphism \[P_i\longrightarrow
P_i/G_i^0=P_i/(\mathrm{ker}\,G_i\longrightarrow
G_{i_0}^\mathrm{et})=X',\quad i\ge i_0\]
gives $P_i\longrightarrow X'$ the structure of a torsor over $X'$
under the structure group $G_i^0$. Moreover, since
$\Gamma(P_i,\mathcal O_{P_i})=k$, $P_i$ cannot be trivial over $X'$
unless $G_i^0=\{e\}$. Employing  the $X'$--morphisms $Y\longrightarrow P_i$, we see, using Claim
B proved in part (1), that $G_i^0=\{e\}$. This means that
$G=G_{i_0}^\mathrm{et}$.
\end{proof}

%%%%%%%%%%%%%%%%%%%%%%%%%%%%%%%%%%%%%%%%%%%%%%%
\section{Finiteness of $G(Y/X)$, reducedness of the universal torsor and base change
properties}\label{se.c}

As in section \ref{s1}, we let $X$ stand for a projective and smooth variety and $f:Y\longrightarrow X$ for a proper surjective morphism from a (proper) variety $Y$. We  also choose a $k$--rational point $x_0:\mathrm{Spec}(k)\longrightarrow X$.

\subsection{An instance where $G(Y/X)$ is not finite and the universal torsor is not reduced}\label{18.05.2011--1} Let $G(Y/X)$ be the affine
fundamental group scheme associated to the Tannakian category
\[\mathcal T_Y(X)\]  by means of the fiber functor $x_0^*:\mathcal
T_Y(X)\longrightarrow \text{($k$--mod)}$. If $V$ is an object of
$\mathcal T_Y(X)$ which is \emph{stable} as a vector bundle (all
vector bundles in $\mathcal T_Y(X)$ are semistable of slope zero
\cite[Proposition 2.2]{tef}), the representation of $G(Y/X)$ obtained from
$V$ must be \emph{irreducible}. Since a finite group scheme only has
\emph{finitely} many isomorphism classes of irreducible
representations --- these are all Jordan--H\"older components of the
right regular representation \cite[3.5]{waterhouse} --- we have a proved
the following lemma.

\begin{lem}\label{finite_stable}If there are infinitely many non--isomorphic 
stable vector
bundles in $\mathcal T_Y(X)$, then the group scheme $G(Y/X)$ is not
finite.
\end{lem}

The existence of infinitely many stable bundles in $\mathcal T_Y(X)$ also causes the following particularity. 

\begin{prp}\label{non_reduced}
Assume that there are infinitely many
non--isomorphic stable vector bundles in $\mathcal T_Y(X)$. Then
there exists a finite quotient $G_0$ of $G(Y/X)$ and a $G_0$--torsor over
$X$, call it $P_0$, such that

(1) $\Gamma(P_0,\mathcal O_{P_0})=k$ and

(2) the scheme $P_0$ is not reduced.

Moreover, in this case, the universal torsor $\widetilde
X\longrightarrow X$ for the fundamental group scheme
$\Pi^\mathrm{EF}(X,x_0)$ is not reduced as a scheme.
\end{prp}
\begin{proof}Let \[G(Y/X):=G=\varprojlim G_i\, ,\]
where each $G_i$ is a finite group--scheme and the transition
morphisms $G_j\longrightarrow G_i$ are faithfully flat, just as in
the proof of Theorem \ref{main_separable1}. By Lemma \ref{finite_stable} and the 
assumption, $G$ is not a finite group scheme. We will show that the
conclusion of the statement holds under the extra assumption that the
group schemes $G_i$ are all local. The general case can be obtained
from this one as in the proof of Theorem \ref{main_separable1}. Let
$P\longrightarrow X$ be the universal $G$--torsor associated to
$\mathcal T_Y(X)\subset\mathbf{EF}(X)$ via the constructions in
\cite[\S 2]{nori}. The torsor $P$ gives rise to $G_i$--torsors
\[
\psi_i:P_i=P\times^GG_i\longrightarrow X.
\]
Due to \cite[Proposition 3, p. 87]{nori-thesis}, we have
$\Gamma(P_i,\mathcal O_{P_i})=k$. Since $G_i$ is a local group
scheme, for any field extension $K/k$, the map
$\psi_i(K):P_i(K)\longrightarrow X(K)$ is bijective, by [EGA $I$,
3.5.10, p. 116] $\psi_i$ induces a bijection on the corresponding
topological spaces. Hence, $\psi_i$ is a homeomorphism and it follows
that $P_i$ is irreducible for each $i$. We assume that each $P_i$ is
also reduced. Proceeding as in the proof of Theorem
\ref{main_separable1} (see Claim A), there exists a $X$--morphism
$h:Y\longrightarrow P_i$ for each $i$. This bounds
$\deg\,\psi_i=\mathrm{rank}\,G_i$ by above and leads to a
contradiction with the assumption that $G$ is not finite.

The proof of the last statement is a direct consequence of what we
just proved together with \cite[Proposition 3]{nori-thesis} and [EGA
$IV_3$, 8.7.2].
\end{proof}

In view of Lemma \ref{finite_stable} and Proposition \ref{non_reduced}, we can use \cite{pauly} to
give an example of a smooth curve $X$ having two extraordinary features: (1) there exists a finite morphism $Y\longrightarrow X$ such that $G(Y/X)$ is \emph{not} finite and (2)
the universal torsor
$\widetilde X$ for the fundamental group scheme
$\Pi^\mathrm{EF}(X,x_0)$ is \emph{not} reduced. Indeed, let $X$ be
the smooth curve constructed in \cite[(3.1) and Proposition
4.1]{pauly}: it is a smooth projective curve defined by a single
explicit equation in $\mathds P^2_k$; here $k$ is any field of
characteristic two. Let $f:Y\longrightarrow X$ be the fourth
power of the Frobenius morphism (so $Y$ is isomorphic to $X$ a scheme). Pauly \cite[Proposition
4.1]{pauly} constructs a locally free coherent sheaf over $X\times
S$, where $S$ is a positive dimensional $k$--scheme, such that for
every $s\in S(k)$, the vector bundle $\mathcal E|X\times\{s\}$ is \emph{stable} 
and
$f^*(\mathcal E|X\times\{s\})$ is trivial. Furthermore, for two
different points $s,t\in S(k)$, the sheaves $\mathcal E|X\times\{s\}$
and $\mathcal E|X\times\{t\}$ are not isomorphic. In other words,
there are infinitely many isomorphism classes of stable vector bundles
of fixed rank satisfying the condition that the pullback by $f$ is 
trivial. By Lemma \ref{finite_stable}, the affine group scheme $G(Y/X)$ is not finite. From Proposition
\ref{non_reduced}, it follows also that the universal torsor $\widetilde
X\longrightarrow X$ is not reduced.

\begin{rmk}
In \cite[Remark 2.4]{ehs} the reader can find an example of an
$\alpha_p$--torsor over a reduced variety which is not reduced. The
example we have just given shows that the situation can be bad even
if the ambient variety is smooth.
\end{rmk}

\subsection{A link between the quantity of $F$--trivial vector bundles and the universal torsor}We assume that $k$ is of positive characteristic,
and let $F:X\longrightarrow X$ be the absolute Frobenius morphism. Define 
\[
S(X,r,t)=\left\{\begin{array}{ll}\text{isomorphism classes of \emph{stable} 
vector bundles of rank $r$} \\ \text{ on $X$, whose pull--back by $F^t$ is 
trivial} \end{array} \right\}
\]
(Here we refrain from using the terminology $F$--trivial, since there is a 
question of stability which is not constant in the literature \cite{pauly}, 
\cite{mehta-sub}.)
In their study of base change for the local
fundamental group scheme and these bundles, Mehta and Subramanian \cite{mehta-sub}
showed the following.

\begin{thm}[\cite{mehta-sub}, Theorem, p. 208] Let $X$ be
a smooth projective variety over
  $k$. The following are equivalent:

(a) For any algebraically closed extension $k'/k$, any pair $r,t\in\mathds N$ 
and any $E'\in S(X\otimes_kk';r,t)$, there exists  a
vector bundle over $X$ and an isomorphism $E\otimes_kk'\cong
E'$.

(b) For any two given $r,t\in\mathds N$, $S(X;r,t)$ is finite.

(c) The local fundamental group scheme of $X\otimes_kk'$ is obtained
from the local fundamental group scheme of $X$ by base change.
\end{thm}

For the definition of the local fundamental group scheme, the reader
should consult \cite{mehta-sub}. In Proposition \ref{non_reduced}
we have shown
that
\[
\left\{\begin{array}{lll}\text{The universal torsor for the } \\
\text{ fundamental group } \\ \text{scheme is a reduced scheme} 
\end{array}\right\}\Longrightarrow
\text{\{Condition (b) in the above theorem holds.\}}
\]

As Vikram Mehta made us realize, the reverse implication need not be true and 
the arguments to follow are due to him. To construct a counter--example, we 
consider an abelian threefold $A$ and $\iota:X\hookrightarrow A$ a closed smooth 
surface defined by intersecting $A$ with a hyperplane section of high degree in 
some projective embedding $A\hookrightarrow \mathds P^N$. By ``Lefschetz's 
Theorem'' \cite[Theorem 1.1]{biswas--lef}, we have an isomorphism 
\[
\Pi^\mathrm{EF}(\iota):\Pi^{\mathrm{EF}}(X,x_0)\stackrel{\cong}{\longrightarrow} 
\Pi^{\mathrm{EF}}(A,x_0)
\]
so that, if $B\longrightarrow A$ is a pointed torsor under a finite group scheme 
with the property of being ``Nori reduced''  \cite[Proposition 3, p. 
87]{nori-thesis}, i.e. \[H^0(B,\mathcal O_B)=k;\]
 the same can then be said about the restriction of $B$ to $X$. Using the 
torsors 
\[
{[p]}:A\longrightarrow A\quad \text{($[p]$ is multiplication by $p$),}
\] 
we see that $X$ admits a ``Nori reduced'' torsor under a finite group scheme 
which is not reduced \emph{as a scheme}. (This follows from the factorization  $[p]=VF$ and the fact that $F^{-1}(Z)$ is never reduced if $Z\subseteq A$ is a proper closed sub--scheme.) By another application of the 
``Lefschetz's Theorem'' \cite[Theorem 1.1]{biswas--lef}, we obtain a bijection 
\[ 
S(A,r,t)\stackrel{\sim}{\longleftrightarrow} S(X,r,t). 
\]
Since the iteration of the Frobenius morphism $F^t_A:A\longrightarrow A$ sits in 
a commutative diagram 
\[
\xymatrix{  A\ar[r]^{[p^t]}\ar[dr] & A   \\  & A\ar[u]_{F_A}, }
\]
if ${F_A^t}^*E$ is trivial, then $[p^t]^*E$ is likewise; consequently, we obtain 
an injection 
\[
S(A,r,t)\hookrightarrow \left\{\begin{array}{ll}\text{isomorphism classes of 
simple}  \\ \text{representations of rank $r$ of 
$\ker\,[p]$}\end{array}\right\}.
\] This entails that $S(X,r,t)$ is always a finite set and we arrive at the 
desired counter--example to the above highlighted implication.

\begin{rmk}[Made after completion]\label{parameswaran_remark}In a recent 
discussion, Parameswaran called our attention to a simpler proof of the fact 
that a vector bundle $E$ on $X$ which becomes trivial after being pulled back by 
a finite morphism from a smooth and projective variety $f:Y\longrightarrow X$ in 
fact comes from a representation of the etale fundamental group of $X$ (compare 
Theorem \ref{main_separable1}). 
The main idea is to use the \emph{algebra} 
\[ 
f_*(\mathcal O_Y)_\mathrm{max} 
\] 
associated to a separable and finite  morphism $f:Y\longrightarrow X$ from a smooth projective variety $Y$ to $X$ (here the subscript ``max'' stands for the maximal semistable subsheaf). That this is in fact an algebra requires a proof and 
the reader is directed to \cite[Lemma 6.4]{BP}. One of the consequences of 
\cite{BP} (which Parameswaran was kind enough to explain to the second author) 
is that $f_*(\mathcal O_Y)_\mathrm{max}$ is the maximal \emph{etale} extension 
of $\mathcal O_X$ inside $f_*\mathcal O_Y$. Together with \cite[Proposition 
6.8]{BP}, the triviality of  $f^*E$ implies the triviality of the pull--back of 
$E$ to the \emph{finite etale} $X$--scheme 
$Y_\mathrm{max}=\mathrm{Spec}\,f_*(\mathcal O_Y)_\mathrm{max}$ and this enough 
to show that $E$ is essentially finite. The reader should also note that in 
\cite[\S 6]{BP}, the framework is such that the domain variety is smooth, which 
is not sufficient to obtain Theorem \ref{main_separable1} directly; but it is 
possible that the methods in \cite{BP} can be extended (for example, to a 
\emph{normal domain variety}) to give another proof of Theorem 
\ref{main_separable1}. 
\end{rmk}

\textbf{Acknowledgements.} We thank the referee for pertinent remarks which made the present text much clearer.

%%%%%%%%%%%%%%%%%%%%%%%%%%%%%%%%%%%%%%%%%%%


\begin{thebibliography}{AAAA}
\bibitem[AM]{antei_mehta} M. Antei and V. Mehta, Vector Bundles over Normal Varieties Trivialized by Finite Morphisms, preprint 2010. arXiv:1009.5234. 


\bibitem[BP]{BP} V. Balaji and A. J. Parameswaran, \emph{An analogue of the 
Narasimhan--Seshadri theorem and some applications}, preprint 2009. 
ArXiv:0809.3765v2.

%\bibitem[BPS]{bps} I. Biswas, A. J. Parameswaran, and S. Subramanian,
%\textit{Monodromy group for a strongly semistable principal bundle
%over a curve}. {Duke Math. Jour.} \textbf{132} (2006), 1--48.

\bibitem[BdS10]{tef} I. Biswas and J. P. dos Santos,
\emph{Vector bundles
trivialized by proper morphisms and the fundamental group scheme}.
Jour. Inst. Math. Jussieu,  Volume 10, Issue	02 (2010),  pp. 225 -- 234. 


\bibitem[BH]{biswas--lef} I. Biswas and Y. Holla, \emph{Comparison of 
fundamental group schemes of a projective variety and an ample hypersurface},
  J. Algebraic Geom. \textbf{16} (2007), no. 3, 547--597.

\bibitem[DM82]{deligne-milne} P. Deligne and J. Milne,
\newblock{\em Tannakian categories}. \newblock{ Lecture Notes in
Mathematics 900, 101--228, Springer-Verlag, Berlin-New York, 1982.}

\bibitem[DG70]{dg} M. Demazure and P. Gabriel, \emph{Groupes
alg\'ebriques}. Masson \& Cie, Paris;
North-Holland Publishing Co., Amsterdam, 1970.

\bibitem[EHS08]{ehs} H. Esnault, P. H. Hai and X. Sun,
\emph{On Nori's fundamental group scheme}. Geometry and dynamics of
groups and spaces, 377--398,
Progr. Math., 265, Birkh\"auser, Basel, 2008.

\bibitem[EGA]{EGA} A. Grothendieck (with the collaboration of J.
Dieudonn\'e). \emph{\'El\'ements de G\'eom\'etrie Alg\'ebrique}.
Publ. Math. IH\'ES \textbf{8}, \textbf{11} (1961); \textbf{17}
(1963); \textbf{20} (1964); \textbf{24} (1965); \textbf{28} (1966);
\textbf{32} (1967). Available at \verb|http://www.numdam.org|.

\bibitem[SGA1]{SGA1} A. Grothendieck et al. \textit{Rev\^etements
\'etales et groupe fondamental}. Lecture Notes in Math. 224,
Springer-Verlag (1971). \verb|http://arxiv.org/abs/math/0206203|.

\bibitem[J87]{J} J. C. Jantzen, \emph{Representations of algebraic
groups}. Pure and Applied Mathematics, 131.
Academic Press, Inc., Boston, MA, 1987.

\bibitem[LS77]{ls}H. Lange and U. Stuhler, \emph{Vektorb\"undel auf
Kurven und Darstellungen der
algebraischen Fundamentalgruppe}. Math. Zeit.
\textbf{156} (1977), 73--84.

%\bibitem[La09]{S1--langer} A. Langer, \textit{On the $S$--fundamental
%group scheme}. \verb|http://arxiv.org/abs/0905.4600v1|. To appear in Annales de 
%l'Institut Fourier.

\bibitem[MS08]{mehta-sub} V. B. Mehta and S. Subramanian,
\textit{Some remarks on the local fundamental group scheme}. Proc.
Indian Acad. Sci. (Math. Sci.) \textbf{118} (2008), 207--211.

\bibitem[No76]{nori} M. V. Nori, \emph{On the representations of the
fundamental
group}. {Compos. Math.} \textbf{33} (1976), 29--41.
\verb|http://www.numdam.org|.

\bibitem[No82]{nori-thesis} M. V. Nori, \emph{Ph.D Thesis}, Proc.
Indian Acad. Sci. (Math. Sci.) \textbf{91} (1982), 73--122.

\bibitem[No83]{nori-abelian} M. V. Nori, \emph{The fundamental group scheme of 
an abelian variety}, Math. Ann. \textbf{263} 263--266.

\bibitem[Pa07]{pauly} C. Pauly, \textit{A smooth counter--example to
Nori's conjecture on the fundamental group scheme}. Proceedings of
the American Mathematical Society \textbf{135} (2007), 2707--2711.

\bibitem[Wa79]{waterhouse}
W. C. Waterhouse, \textit{Introduction to affine group schemes}.
Graduate Texts in Mathematics, 66. Springer-Verlag, New York-Berlin,
1979.
\end{thebibliography}
\end{document}